\documentclass[a4paper,11pt]{article}

\usepackage[utf8]{inputenc}
\usepackage[margin=2cm]{geometry}
\usepackage{amssymb, amsfonts, amsmath,amsthm,bbm}
\usepackage{subcaption}

\usepackage{mathtools}
\usepackage{tikz}
\usepackage{caption}
\usepackage{wrapfig}

\usepackage{newtxtext}

\numberwithin{equation}{section}

\theoremstyle{plain}
\newtheorem{theorem}{Theorem}[section]
\newtheorem{corollary}[theorem]{Corollary}

\newtheorem{lemma}[theorem]{Lemma}
\newtheorem{proposition}[theorem]{Proposition}

\theoremstyle{definition}

\theoremstyle{remark}
\newtheorem{remark}[theorem]{Remark}

\usetikzlibrary{fit,hobby,calc,shapes.arrows,decorations.pathreplacing,shapes,positioning,automata}
\newcommand{\IR}{\mathbb{R}}
\newcommand{\IP}{\mathbb{P}}
\newcommand{\IC}{\mathbb{C}}

\newcommand{\N}{\mathbb{N}}

\newcommand{\E}{\mathbb{E}}
\newcommand{\R}{\mathbb{R}}
\newcommand{\Z}{\mathbb{Z}}

\newcommand{\cL}{\mathcal{L}}

\renewcommand{\L}{\mathbbm{L}}

\newcommand{\C}{\mathbb{C}}

\newcommand{\erm}{\mathrm{e}}
\newcommand{\egaldistr}{\overset{(d)}{=}}
\renewcommand{\bar}[1]{\overline{#1}}
\newcommand{\dd}{\mathrm{d}}

\newcommand{\Var}{\mathbb{V}\mathrm{ar}}

\newcommand{\ind}[1]{\mathbbm{1}_{\left\{#1\right\}}}
\newcommand{\crochet}[1]{\left\langle #1 \right\rangle }

 \title{\bfseries A spiraling elephant random walk}

\author{Lucile Laulin\thanks{Modal’X, Université Paris Nanterre, CNRS UMR 9023, Nanterre France.} \and Bastien Mallein\thanks{Univ Toulouse, INSA Toulouse, CNRS, Institut de Mathématiques de Toulouse, F-31062 Toulouse Cedex 9, France.}{~}\thanks{Institut universitaire de France (IUF).}}
 \title{\bfseries Elephants explore in spirals sometimes}

\begin{document}

\maketitle

\begin{abstract}
We consider in this article an Elephant Random Walk evolving in the plane. Specifically, this is a reinforced stochastic process in which the $n$th step is given by a random rotation of one of the previous steps chosen uniformly at random. We obtain a central limit theorem for this process, which shows that the process follows a randomly rotated logarithmic spiral at large times, with Gaussian fluctuations.
\end{abstract}

\section{Introduction}

The elephant random walk (ERW) is a well-studied reinforcement process, introduced by Schütz and Trimper \cite{ScT} to investigate the influence of memory on random walk dynamics. In $\mathbb{Z}$, the process is constructed as follows. Given a memory parameter $p \in [0,1]$, the process starts from initial position $S_0 = 0$, then moves to $S_1 = 1$ at the first step\footnote{Alternatively, this step could be randomized, but by construction, flipping the initial path as the same effect as flipping the whole trajectory of the walk along the $x$-axis.}. At each subsequent time $n \geq 2$, the walker chooses uniformly at random one of its $n-1$ past steps. It then repeats this step with probability $p$, or with complementary probability $1-p$ moves one step in the opposite direction. In other words, the ERW is defined by induction by
\begin{equation}
  \label{eqn:defERW1}
  S_0 = 0, \quad S_1 = 1, \quad S_{n} = S_{n-1} + R_{n} (S_{I_n} - S_{I_n - 1}) \text{ for $n \geq 2$,}
\end{equation}
where $(R_n, n \geq 2)$ are i.i.d. Radamacher random variables with parameter $p$ that are
further independent from the sequence $(I_n, n \geq 2)$ of independent random variables with $I_k$ uniformly distributed on $\{1,\ldots,k-1\}$.

Provided that $p>1/2$, this dynamic can be rephrased as follows. At each step, the walker remembers, with probability $2p-1$, one of its past steps chosen uniformly at random, and repeats this step. With complementary probability $2(1-p)$, it performs an independent step equal to $1$ or $-1$. An analogous definition can be taken when $p<1/2$ considering $1-2p$ as the probability of performing the opposite remembered step. The parameter $p$ is usually referred to as the memory parameter of the ERW, but we see from the previous rephrasing the effective memory is more accurately described by $a := 2p-1$. The value $a=0$ (i.e. $p=1/2$) corresponds to the absence of memory, while $|a|$ measures the strength of the dependence on past increments. When $a>0$, the walker tends to repeat the chosen past step (“positive” memory), whereas for $a<0$ it tends to do the opposite (“negative” memory). In both cases, the memory is strong when $|a|$ is close to $1$.

From this perspective, the walk appears to exhibit long-range dependence, as the walk tends for example to reinforce previous directions when $a>0$ -- the ERW is a special case of \textit{step-reinforced random walk}. Consequently, the walker’s trajectory can deviate significantly from a simple random walk, leading to superdiffusive behavior when $a>1/2$ while for $a\in(-1,1/2)$ the process remains diffusive but with different variance compared to the classical case.

A natural generalisation for this process would be to define a version of the elephant random walk in dimension $d$. A $\mathbb{Z}^d$ version of this model was first introduced in \cite{BercuLaulin2019}, while the works \cite{CuL,Qin} further investigate its properties, notably recurrence and transience.
The process considered was the following: at each time $n$, after selecting uniformly at random a previous step, the walker either moves according to this step with probability $p$, or with probability $1-p$ moves towards one of the $2d-1$ other neighbour of $S_n$ chosen uniformly at random. It can be shown in this situation that the ERW undergoes a phase transition when the memory parameter crosses $p_c = \frac{2d+1}{4d}$. If $p \leq p_c$, the ERW is in the diffusive regime, while it drifts at some sub-linear rate if $p > p_c$. This behaviour is therefore very close to the one observed in dimension $1$.

The objective of this note is to show that a larger variety of behaviours may be observed in dimension greater than $1$, by modifying the above described evolution of the process: instead of moving to a uniformly chosen neighbour of $S_n$ with probability $1-p$, we consider a process in which, at each steps, the walker moves according to a random rotation of the previously chosen step. More precisely, let $(R_n, n \geq 2)$ be i.i.d. random rotations of $\R^d$ (i.e. random variables in $\mathcal{O}(\R^d)$) independent of the sequence $(I_n, n \geq 2)$ defined above, we then define the ERW in $\R^d$ as
\begin{equation}
  \label{eqn:defERWd}
    S_0 = (0,\ldots, 0), \quad S_1 = (1,0,\ldots, 0) \quad \text{and} \quad S_{n} = S_{n-1} + R_n(S_{I_{n}} - S_{I_{n}-1}).
\end{equation}
This model clearly generalizes the one considered above, since we no longer restrict the law of $R_n$ to satisfy $\IP(R_n = \mathrm{Id}) = p$ and $\cL\big(R_n | R_n \neq \mathrm{Id}\big)$ uniformly selected among the non-identity rotations that preserve the canonical base of $\R^d$.

In the present article, we restrict our attention to elephant random walks in the plane. For our purpose, it will be convenient to identify $\R^2$ with the complex plane $\C$ in the usual fashion, and we define the ERW in $\C$ as the following complex-valued stochastic process. Let $(\theta_n, n \geq 1)$ be i.i.d. $[0,2\pi)$-valued random variables, further independent from the sequence $(I_n, n \geq 2)$ defined above, we set
\begin{equation}
  \label{eqn:defERWC}
  S_0 = 0, \quad S_1 = 1 \quad \text{and}\quad S_n= S_{n-1} + (S_{I_n} - S_{I_n-1}) \erm^{i \theta_n}.
\end{equation}
Some sample trajectories of $S$ are represented in Figure~\ref{fig:ex} when $\theta_1$ is a.s. a constant.

\begin{figure}[ht]
\centering
  \begin{subfigure}[t]{.3\linewidth}
    \centering\includegraphics[width=.95\linewidth]{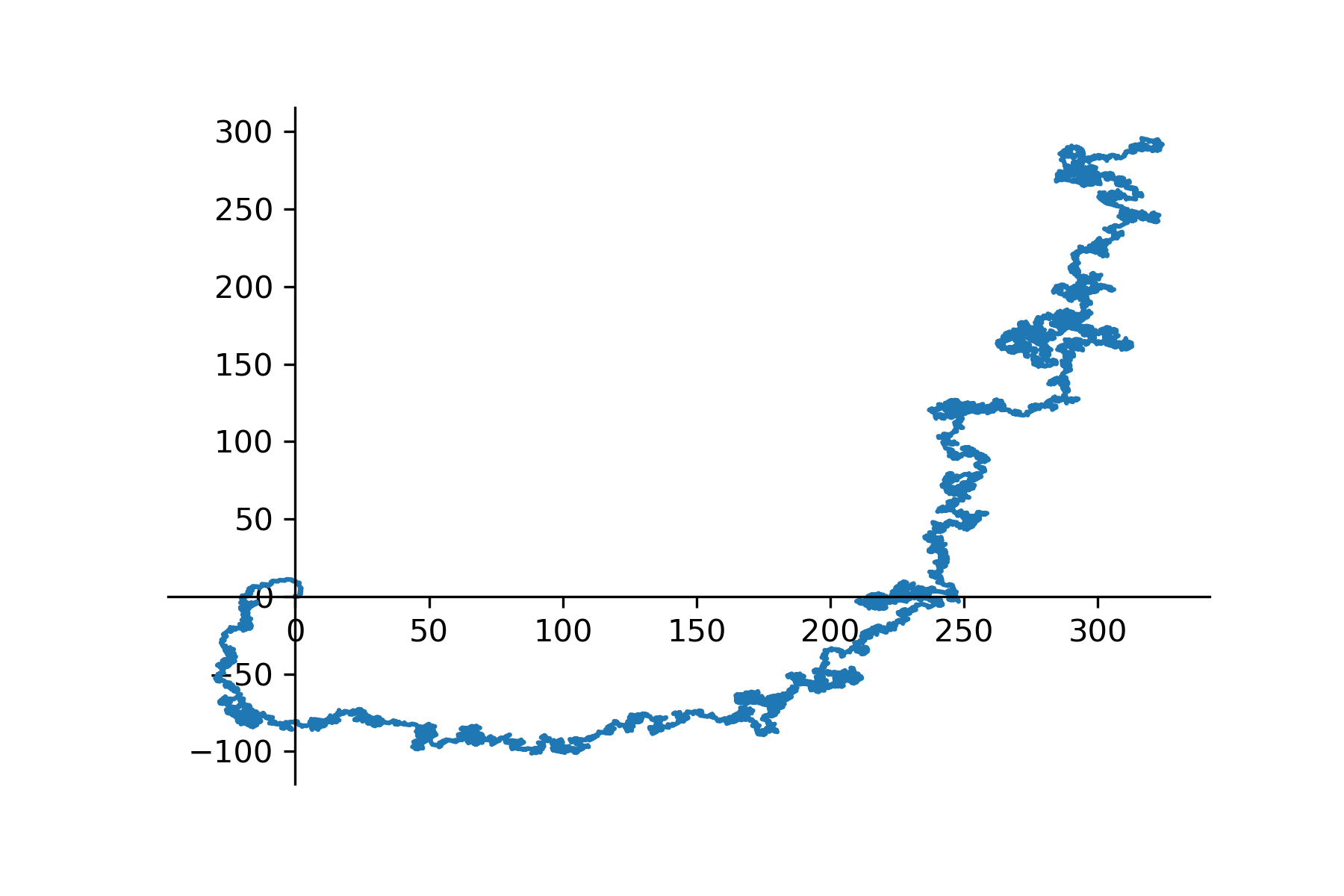}
    \caption{$\theta_1= \frac{\pi}{3} - 0.1$ a.s.}
  \end{subfigure}
  \begin{subfigure}[t]{.3\linewidth}
    \centering\includegraphics[width=.95\linewidth]{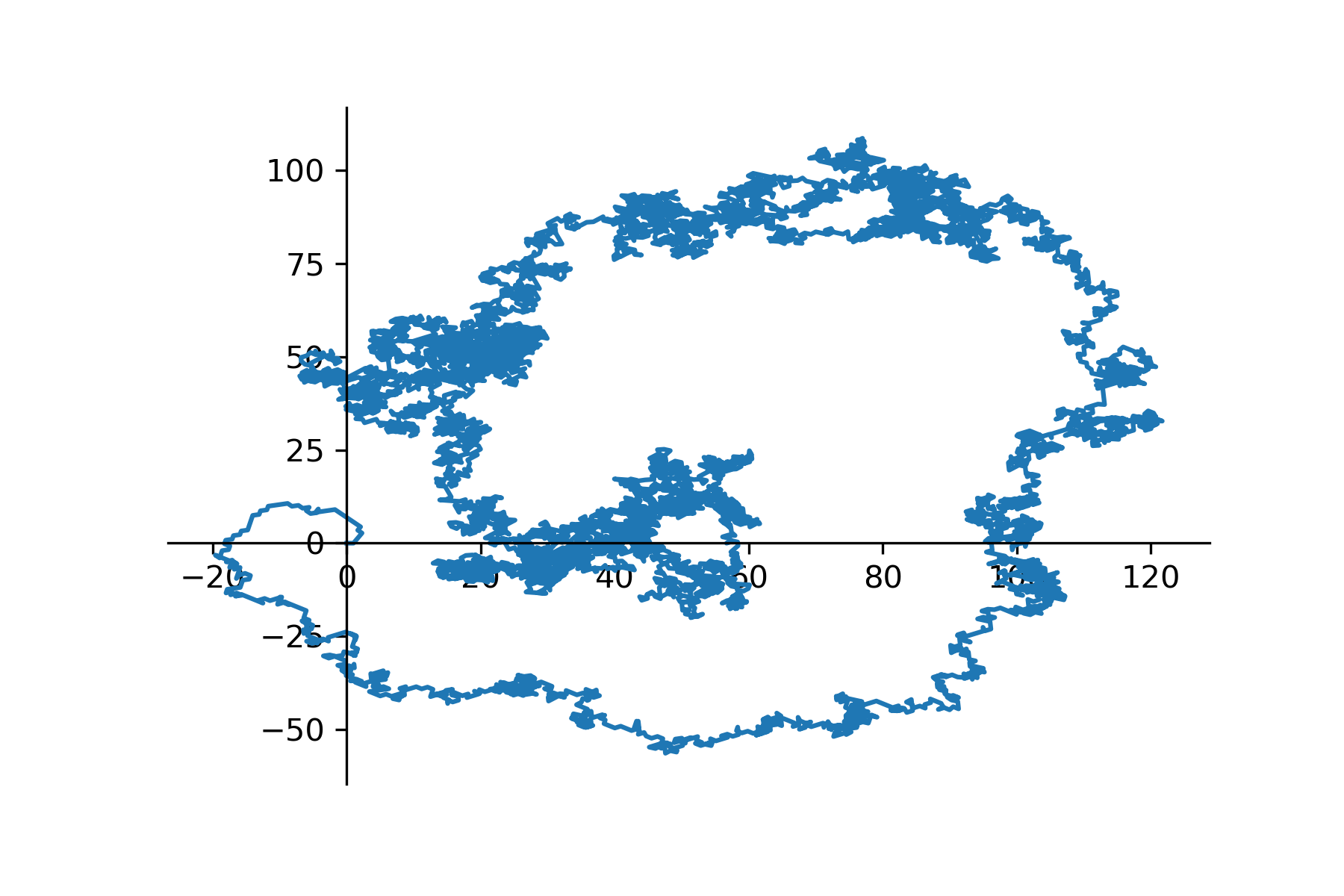}
    \caption{$\theta_1 = \frac{\pi}{3}$ a.s.}
  \end{subfigure}
  \begin{subfigure}[t]{.3\linewidth}
    \centering\includegraphics[width=.95\linewidth]{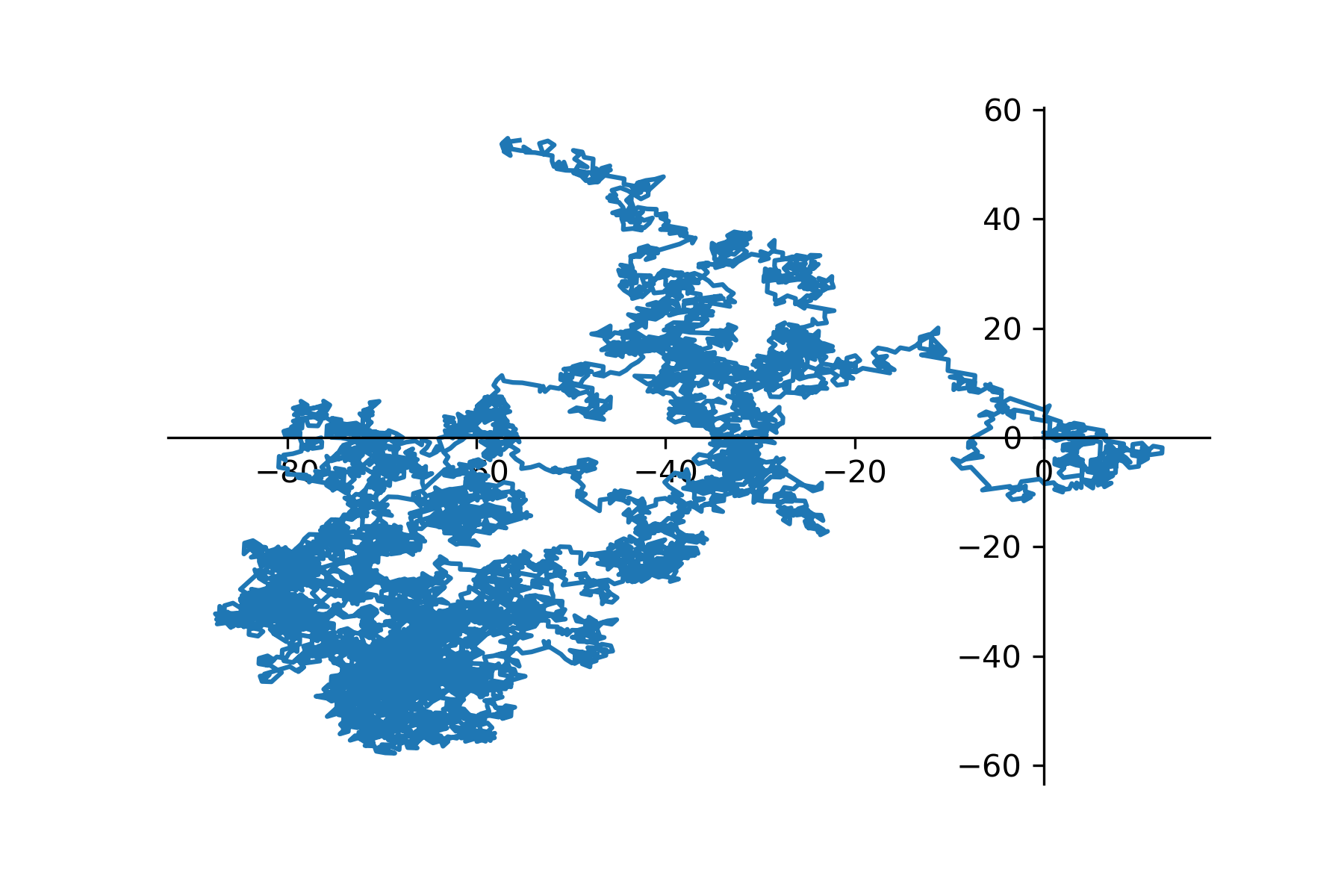}
    \caption{$\theta_1 = \frac{\pi}{3} + 0.1$ a.s.}
  \end{subfigure}
\caption{Sample path of $1000$ steps of an elephant random walks in $\R^2$ as defined in \eqref{eqn:defERWC}, in which $\theta_1$ is a.s. a constant.}
\label{fig:ex}
\end{figure}

The main result of this article is a central limit theorem for the the ERW $(S_n, n \geq 1)$. We show that the asymptotic behaviour of $S_n$ depends mainly of
\begin{equation}
  \Phi_k := \E(\erm^{i k \theta_1}), \quad \text{for $k \geq 1$},
\end{equation}
the Fourier coefficient of the law of $\theta_1$ and, more specifically,  of $\Phi_1$, that plays a role analogue to the memory parameter $a$ in dimension $1$. Indeed, when $d=1$ and $\theta_1$ takes values in $\{0,\pi\}$ with $\IP(\theta_1=0)=p$, one recovers $\Phi_1=2p-1=a$, the usual memory parameter. For $d>1$, $\Phi_1$ describes the average rotation applied to the remembered steps, therefore $|\Phi_1| \in [0,1]$ may roughly be thought of as the strength of that memory, minimal for the random walk $(\Phi_1=0)$, and maximal in the situation represented in Figure~\ref{fig:ex} in which $\Phi_1 \in \mathbb{S}^1$.

To avoid considering a unidimensional system, we always work in this article under the assumption
\begin{equation}
  \label{eqn:Nunidimensional}
  \IP(\theta \in \{0,\pi\}) < 1, \quad \text{or equivalently}\quad \Re(\Phi_2) < 1.
\end{equation}
For $\sigma^2 > 0$, we denote by $\mathcal{N}_\C(0,\sigma^2)$ the centred complex normal distribution, such that $N$ has law $\mathcal{N}_\C(0,\sigma^2)$ if $\Re(N)$ and $\Im(N)$, its real and imaginary parts, are independent real-valued centred normal random variables with variance $\sigma^2/2$.

The main result of this article is that the ERW in $\C$ exhibits a superdiffusive behaviour if $\Re(\Phi_1)>1/2$, while it remains diffusive when $\Re(\Phi_1)<1/2$. In addition, we show that when superdiffusive, the path of the ERW scales towards a randomly rotated logarithmic spiral.

\begin{theorem}
\label{thm:main}
Let $(\theta_n)$ be i.i.d. random variables on $[0,2\pi)$ satisfying \eqref{eqn:Nunidimensional}, and $(S_n, n \geq 1)$ the ERW defined by \eqref{eqn:defERWC}.
\begin{itemize}
  \item If $\Re(\Phi_1) < 1/2$, then writing $\sigma^2 = \frac{1}{1 - 2 \Re(\Phi_1)} \in (1/3,\infty)$, we have
  \begin{equation}
   \label{eqn:mainSub}
    \lim_{n \to \infty} \frac{S_n}{n^{1/2}} = \mathcal{N}_\C(0,\sigma^2) \quad \text{ in law}.
  \end{equation}
  \item If $\Re(\Phi_1)= 1/2$, then
  \begin{equation}
    \label{eqn:mainCrit}
    \lim_{n \to \infty} \frac{S_n}{(n \log n)^{1/2}} = \mathcal{N}_\C(0, 1)\quad \text{ in law}.
  \end{equation}
  \item If $\Re(\Phi_1)> 1/2$, there exists a $\C$-valued random variable $W$ such that, writing $\sigma^2 =\frac{1}{2 \Re(\Phi_1) - 1} \in(1,\infty)$, we have
  \begin{equation}
    \label{eqn:mainSup}
    \lim_{n \to \infty} S_n\erm^{- \Phi_1 \log n} = W \quad \text{a.s. \quad and} \quad \lim_{n \to \infty} \frac{S_n  - \erm^{\Phi_1 \log n} W}{n^{1/2}} = \mathcal{N}_\C(0,\sigma^2) \quad \text{ in law}.
  \end{equation}
\end{itemize}
\end{theorem}

In order to provide variety to this article, we propose to prove Theorem~\ref{thm:main} in two different ways. We study the regime $\Re(\Phi_1) \leq 1/2$ in Section~\ref{sec:diffusive} using Lindeberg's central limit theorem. The key observation in this section is that
\begin{equation}
  \label{eqn:defMartingale}
  \left( \frac{S_n}{\prod_{j=1}^n \left( 1 + \frac{\Phi_1}{j} \right)}, n\geq 1 \right) \text{ is a martingale,}
\end{equation}
its asymptotic behaviour can be deduced from studying its quadratic variation. Up to the difficulty of dealing with a complex martingale, the argument there follows closely standard methods in dimension 1 or $d$ \cite{BaB,BercuLaulin2019}.

We then turn to the regime $\Re(\Phi_1)> 1/2$ in Section~\ref{sec:spiraling}. We use there the known connection between ERW and P\'olya urns, which was notably exploited by Baur and Bertoin in \cite{BaB} to obtain a functional central limit theorem, as well as the connection between P\'olya urns and multitype Galton-Watson processes, that was exploited by Janson~\cite{Jan} to obtain asymptotics on the behaviour of these urns. These results cannot be applied immediately as we do not necessarily have a finite number of colors\footnote{Understand here directions that steps of the ERW might align with.}. P\'olya urns with infinite number have been already studied in the literature \cite{JMV} however, considering that we are only interested with a single martingale of that process, the analysis can be simplified. Specifically, borrowing arguments from the study of the additive martingales of branching random walks at complex parameters \cite{IKM} will allow us to conclude.

Before turning to the proofs though, we begin by stating explicitly the behaviour of the ERW in $\Z^2$. This process is parametrized by four positive numbers $p,q,r,s$ such that $p+q+r+s=1$.  These parameters correspond respectively to the probability for a walker oriented according to one of its uniformly sampled previous steps to move forward (repeating the step), to the left (making a rotation of angle $\pi/2$), backward (making a rotation of angle $\pi$) or to the right (making a rotation of angle $3\pi/2$). The case $q=r=s=(1-p)/3$ reduces to the previously studied standard ERW in the plane.

In this model, remark that $\Phi_1= p - r + i (q-s)$, therefore Theorem~\ref{thm:main} implies the following behaviour for $(S_n, n \geq 1)$, depending on the sign of $p-r-1/2$.
\begin{corollary}
Let $(S_n, n \geq 1) $ be the ERW in $\Z^2$ with parameters $p,q,r,s$ defined above.
\begin{itemize}
  \item If $p < r+1/2$, then $\lim_{n \to \infty} \frac{S_n}{n^{1/2}} = \mathcal{N}(0, \frac{1}{2 - 4 (p-r)} I_2)$ in law.
  \item If $p=r+1/2$, then $\lim_{n \to \infty} \frac{S_n}{(n\log )^{1/2}}= \mathcal{N}(0, \frac{1}{2} I_2)$ in law.
  \item If $p > r+1/2$, then there exists a random vector $W \in \R^2$ such that
  \begin{multline*}
    \lim_{n \to \infty}  \frac{\mathcal{R}_{-(q-s)\log n}S_n}{n^{p-r}} = W \quad \text{a.s.} \\ \text{and}\quad \lim_{n \to \infty} \frac{S_n - n^{p-r} \mathcal{R}_{(q-s)\log n}W}{n^{1/2}} = \mathcal{N}(0,\frac{1}{4(p-r)-2}I_2) \quad \text{ in law},
  \end{multline*}
  writing $\mathcal{R}_\theta$ for a rotation of angle $\theta$.
\end{itemize}
\end{corollary}

Finally, remark that the ERW in $\Z^2$ can be seen as a reinforced version of the Markov chains studied by Lopusanschi and Simon in \cite{LoS}, studying the asymptotic behaviour of the Lévy area of this path would therefore be of interest, especially studying the convergence of the ERW under a rough path topology.

\medskip
We now turn to the proof of Theorem~\ref{thm:main} in the two following sections. The proof of \eqref{eqn:mainSub} is given in the following section using Lindenberg's central limit theorem, with a sketch of the proof of \eqref{eqn:mainCrit}. The proof of \eqref{eqn:mainSup} is completed in Section~\ref{sec:spiraling}, using Poissonization methods and the analysis of a branching process.

\section{Elephant Random Walk in the diffusive regime}
\label{sec:diffusive}

We denote by $(S_n, n \geq 1)$ an ERW defined by \eqref{eqn:defERWC}, and we assume in this section that
\begin{equation}
  \label{eqn:cdSub}
  \Re(\Phi_1) = \Re\left(\E\left( \erm^{i\theta_1} \right)\right) < \frac{1}{2}.
\end{equation}
We prove in this section that
\[
  \lim_{n \to \infty} \frac{S_n}{n^{1/2}} = \mathcal{N}_\C\left(0,\frac{1}{1- 2 \Re(\Phi_1)}\right) \quad \text{ in distribution.}
\]
The proof of \eqref{eqn:mainCrit} follows a similar path with minor modifications, and will be sketched at the end of the section. Let us mention that \eqref{eqn:mainSup} also could have be obtained using the same techniques as developed here.

As mentioned in the introduction, this proof relies on applying a martingale central limit theorem. We first show in the next section that up to scaling, the ERW in $\C$ may be thought of as a complex martingale. We then compute the quadratic variations of this martingale in Section~\ref{subsec:quadVar} before applying Lindeberg's central limit theorem to complete the proof in Section~\ref{subsec:lindeberg}.

\subsection{Discrete martingale structure}

Recall that $\Phi_k = \E(\erm^{i k \theta_1})$. We define, for $n \in \N$ and $k \geq 1$,
\begin{equation}
  \label{eqn:centring}
  a^{(k)}_n = \prod_{j=1}^{n-1} \left( 1 + \frac{\Phi_k}{j} \right),
\end{equation}
and we write for simplicity $a_n = a^{(1)}_n$. As a first step, we show that $M = (S_n/a_n, n \geq 1)$ is a complex martingale. However, with future application in mind, we prove this result for the ERW $S_n^{(k)}$, which is defined as in \eqref{eqn:defERWC}, replacing $\theta_1$ by $k \theta_1$.

\begin{lemma}
\label{lem:martingale}
Let $k \geq 1$, we denote by
\[
  S_{n}^{(k)} = \sum_{j=0}^{n-1} (S_{j+1} - S_j)^k.
\]
The process $(S_n^{(k)}/a_n^{(k)}, n \geq 1)$ is a martingale. Moreover, there exists $C_k > 0$ such that
\[
  \E\left( \left| \frac{S_n^{(k)}}{a_n^{(k)}}\right|^2 \right) \sim \begin{cases}  C_k & \text{ if } \Re(\Phi_k) > 1/2\\
  C_k \log n & \text{if} \Re(\Phi_k) = 1/2\\
  C_k n^{1 - 2 \Re(\Phi_k)} & \text{ if } \Re(\Phi_k) < 1/2 \end{cases} \quad \text{ as } n \to \infty.
\]
\end{lemma}

\begin{proof}
We write $M_n^{(k)} = S_n^{(k)}/a_h^{(k)}$. By definition, we have $|S^{(k)}_{i+1} - S_i^{(k)}| = 1$, therefore $|M_n^{(k)}| \leq n/a_n^{(k)}$ a.s. Next, by definition, writing $\mathcal{F}_n = \sigma(S_1,\ldots,S_n)$, we have
\begin{align*}
  \E\left( S^{(k)}_{n+1} | \mathcal{F}_n \right) &= \E\left( S^{(k)}_n + (S_{I_{n+1}} - S_{I_{n+1}-1})^k  \erm^{i k \theta_{n+1}} \middle| \mathcal{F}_n\right)\\
  &= S_n^{(k)} + \frac{1}{n} \sum_{i = 1}^n (S_{i+1}^{(k)} - S_i^{(k)}) \E(\erm^{i k \theta_{n+1}})\\
  &= S^{(k)}_n (1 + \frac{\Phi_k}{n}) \quad \text{a.s.}
\end{align*}
As a result, we have $\E(M^{(k)}_{n+1}|\mathcal{F}_n)=  M^{(k)}_n$ a.s., showing that $M^{(k)}$ is a martingale.

We now compute the second moment of $S^{(k)}$, observing that
\begin{align*}
  \E\left( \left| S^{(k)}_{n+1}\right|^2 \middle| \mathcal{F}_n\right)
  &= \E\left( (S_n^{(k)} + \Delta S_n^{(k)}) (\bar{S_n^{(k)} + \Delta S_n^{(k)} }) \middle| \mathcal{F}_n\right)\\
  &= \left| S_n^{(k)} \right|^2 + 2 \Re\left(S_n^{(k)} \bar{\E\left(\Delta S_n^{(k)}\middle| \mathcal{F}_n\right)} \right) + \E\left( \left| \Delta S_n^{(k)}\right|^2 \middle| \mathcal{F}_n\right),
\end{align*}
where we write $\Delta S^{(k)}_n= S^{(k)}_{n+1} - S^{(k)}_n = (S_{n+1} - S_n)^k$. Remark that $|\Delta S_n^{(k)}| = 1$ a.s. and that
\[
  \E\left(\Delta S_n^{(k)}\middle| \mathcal{F}_n\right) = \frac{\Phi_k}{n} \sum_{i=1}^n \Delta S_{i}^{(k)} = \Phi_k\frac{S_n^{(k)}}{n} \text{ a.s.,}
\]
from which we deduce that
\begin{align*}
  \E\left( \left| S^{(k)}_{n+1} \right|^2  \right) &= \E\left( \left| S_n^{(k)} \right|^2 \right) + \frac{2}{n} \Re\left( \bar{\Phi_k}\E\left( S_n^{(k)} \bar{S_n^{(k)}} \right) \right) + 1\\
  &= \E\left(  \left| S_n^{(k)} \right|^2  \right) \left(1 + \frac{2}{n}\Re(\Phi_k)\right) + 1.
\end{align*}
As a result, setting $u_n = \E\left( \left| \frac{S^{(k)}_n}{a_n^{(k)}} \right|^2 \right)$, we observe that $(u_n)$ solves the recursion equation
\begin{align*}
  u_{n+1} &= \frac{|a^{(k)}_n|^{2}}{|a^{(k)}_{n+1}|^{2}} u_n \left(1 + \frac{2}{n}\Re(\Phi_k)\right)  + \frac{1}{|a_{n+1}^{(k)}|^2}\\
  &= u_n \frac{\left( 1 + \frac{2 \Re(\Phi_k)}{n} \right)}{|1 + \frac{\Phi_k}{n}|^2}   + \frac{1}{|a_{n+1}^{(k)}|^2},
\end{align*}
with $u_1 = 1$. As $\frac{1  + \frac{2 \Re(\Phi_k)}{n}}{|1 + \frac{\Phi_k}{n}|^2} = \frac{1}{1 + \frac{|\Phi_k|^2}{n^2(1+ \frac{2\Re(\Phi_k)}{n})}}$, we observe that $u_{n+1} - u_n - \frac{1}{|a_{n+1}^{(k)}|^2} = O (u_n/n^2)$ as $n \to \infty$.

We denote by $\Gamma$ Euler's Gamma function defined on $\C \backslash \Z_-$, and observe that we can rewrite
\begin{align*}
   a^{(k)}_n &= \prod_{j=1}^{n-1} \left( 1 + \frac{\Phi_k}{j} \right) = \prod_{j=1}^{n-1}\frac{j+\Phi_k}{j} \frac{\Gamma(j+\Phi_k)}{\Gamma(j)} \frac{\Gamma(j)}{\Gamma(j+\Phi_k)}\\
  &= \prod_{j=1}^{n-1} \frac{\Gamma(j + \Phi_k+ 1)}{\Gamma(j+\Phi_k)} \frac{\Gamma(j)}{\Gamma(j + 1)} = \frac{\Gamma(n+\Phi_k)}{\Gamma(n)\Gamma(\Phi_k)},
\end{align*}
from which we deduce there exists $c_k > 0$ such that
\begin{equation}
 \label{eqn:data}
 a_n^{(k)} \sim C_k n^{\Phi_k}, \text{ and } |a_n^{(k)}|^2 \sim |C_k|^2 n^{2 \Re(\Phi_k)} \text{ as $n \to \infty$.}
\end{equation}
Using that $\Re(\Phi_k) \in (-1,1)$ due to \eqref{eqn:Nunidimensional} we conclude that
\[
  \sum_{j=1}^n \frac{1}{|a_j^{(k)}|^2} \begin{cases}
    = O(1) & \text{ if } \Re(\Phi_k)>1/2\\
    \sim |C_k|^2 \log n & \text{ if } \Re(\Phi_k) = 1/2\\
    \sim \frac{|C_k|^2}{1 - 2 \Re(\Phi_k)} n^{1 - 2 \Re(\Phi_k)} &\text{ if } \Re(\Phi_k) < 1/2
  \end{cases}
  \quad \text{ as $n \to \infty$.}
\]
Solving the recursion for $u$ in all three cases, we conclude the proof.
\end{proof}

\subsection{Convergence rates of quadratic variations}
\label{subsec:quadVar}

In the rest of the section, we pay specific interest to the martingale $M$, defined by $M_n = S_n/a_n$. We study here the quadratic variations of $(M_n)$, defined as follows. For $k \geq 1$, we write $\Delta M_k = M_{k+1} - M_k$ and
\begin{equation}
  \label{eqn:martDefs}
  \begin{split}
    \crochet{M}_n &= \sum_{k=1}^{n-1} \E\left( (\Delta M_k)^2 \middle| \mathcal{F}_{k} \right),\\
    \crochet{M,\overline{M}}_n &= \sum_{k=1}^{n-1} \E\left( \Delta M_k \overline{\Delta M_k} \middle| \mathcal{F}_{k} \right) = \sum_{k=1}^{n-1} \E\left( |\Delta M_k|^2 \middle| \mathcal{F}_{k} \right).
  \end{split}
\end{equation}
We now compute these two quantities and estimate their asymptotic behaviour when $\Re(\Phi_1) < 1/2$.

\begin{lemma}
\label{lem:quad-complex}
Under \eqref{eqn:cdSub}, we have
\begin{equation}
  \crochet{M,\overline{M}}_n \sim \frac{1}{1 - 2 \Re(\Phi_1)} n^{1 - 2 \Re(\Phi_1)} \quad \text{ in probability as $n \to \infty$.}
\end{equation}
\end{lemma}

\begin{proof}
Let $n \geq 1$, we observe that
\begin{align*}
  \Delta M_k &= M_{k+1} - M_k = \frac{S_n + \Delta S_n}{a_{n+1}} - \frac{S_n}{a_n} = \frac{1}{a_{n+1}}\left( S_n + \Delta S_n - \left(1+  \frac{\Phi_1}{n}\right) S_n \right)\\
  &= \frac{1}{a_{n+1}} \left( \Delta S_n - \frac{\Phi_1}{n}S_n \right).
\end{align*}
Therefore,
\begin{align*}
  \E\left( |\Delta M_k|^2 \middle| \mathcal{F}_k \right)
  &= \frac{1}{|a_{k+1}|^2}\E\left( |\Delta S_k|^2 + \left| \frac{\Phi_1}{k} S_k \right|^2 - 2 \Re\left( \frac{\Phi_1}{k} S_k \overline{\Delta S_k}  \right) \middle| \mathcal{F}_k \right)\\
  &= \frac{1}{|a_{k+1}|^2} \left( 1 + \left| \frac{\Phi_1}{k} S_k \right|^2 - 2 \Re\left( \left|\frac{\Phi_1}{k} S_k\right|^2  \right)\right) \text{ a.s.},
\end{align*}
using that $|\Delta S_k| = 1$ and $\E(\Delta S_{k}| \mathcal{F}_k) = \Phi_1\frac{S_k}{k}$ a.s. This yields
\begin{equation*}
  \E\left( |\Delta M_k|^2 \middle| \mathcal{F}_k \right) = \frac{1}{|a_{k+1}|^2}\left( 1 - \left| \frac{\Phi_1}{k} S_k \right|^2\right) \text{ a.s.}
\end{equation*}
As a consequence, we have
\begin{equation}
  \label{eqn:formula}
  \crochet{M,\bar{M}}_n = \sum_{k=1}^{n-1} \frac{1}{|a_{k+1}|^2} -  \sum_{k=1}^{n-1} \frac{1}{|a_{k+1}|^2} \left|\Phi_1\frac{S_k}{k}\right|^2 \quad \text{a.s.}
\end{equation}
We now study the asymptotic behaviour of these two sums.
We deduce from \eqref{eqn:data} that
\begin{equation}
  \label{eqn:asymp}
  v_n := \sum_{k=1}^{n-1} \frac{1}{|a_{k+1}|^2} \sim |C_1|^2 \frac{n^{1 - 2 \Re(\Phi_{1})}}{1 - 2 \Re(\Phi_1)},
\end{equation}
and from Lemma~\ref{lem:martingale} that
\[
  \frac{1}{|a_{n+1}|^2 n^2}\E\left( |S_n|^2\right) \sim_{n \to \infty} |C_1|^2 n^{-(1+ 2\Re(\Phi_1))},
\]
as $n \to \infty$. Therefore, there exists $C > 0$ such that for all $n$ large enough, we have
\begin{equation}
  \label{eqn:intermediate}
  \E\left(  \sum_{k=1}^{n-1} \frac{1}{|a_{k+1}|^2} \left|\Phi_1\frac{S_k}{k}\right|^2\right) \leq C n^{-2\Re(\Phi_1)} = o(v_n).
\end{equation}
It shows that $\lim_{n \to \infty} \frac{1}{v_n}\sum_{k=1}^{n-1} \frac{1}{|a_{k+1}|^2} \left|\Phi_1\frac{S_k}{k}\right|^2 = 0$ in probability, which completes the proof.
\end{proof}

We now turn to the complex quadratic variation of the complex martingale.
\begin{lemma}
\label{lem:quad-mod}
We have $\lim_{n \to \infty} \frac{1}{n^{1- 2\Re(\Phi_1)}} \crochet{M}_n = 0$ in probability.
\end{lemma}

\begin{proof}
The proof follows a similar path to the previous one. We observe that
\begin{align*}
   \E\left( (\Delta M_k)^2 \middle| \mathcal{F}_k \right)
   &= \frac{1}{a_{k+1}^2}\E\left( (\Delta S_k)^2 + \left( \frac{\Phi_1}{k} S_k \right)^2 - 2 \left( \frac{\Phi_1}{k} S_k \Delta S_k  \right) \middle| \mathcal{F}_k \right)\\
   &= \frac{1}{a_{k+1}^2}  \left( \frac{\Phi_2}{k} \sum_{i=0}^{k-1} (\Delta S_i)^2+  \left( \frac{\Phi_1}{k} S_k \right)^2 - 2 \left( \frac{\Phi_1}{k} S_k\right)^2 \right)\\
   &= \frac{1}{a_{k+1}^2} \left( \frac{\Phi_2}{k} S_k^{(2)} - \left( \frac{\Phi_1}{k} S_k \right)^2 \right) \text{ a.s.},
\end{align*}
by definition. As a result, we have
\begin{align*}
  \crochet{M}_n& = \sum_{k=1}^{n-1} \frac{\Phi_2}{a_{k+1}^2}\frac{S_k^{(2)}}{k}  - \sum_{k=1}^{n-1} \frac{\Phi_1^2}{a_{k+1}^2} \left( \frac{S_k}{k} \right)^2 \text{a.s.}
\end{align*}
We now show that these two terms are $o(n^{1 - 2 \Re(\Phi_1)})$ in probability as $n \to \infty$.

Observe that by triangular inequality, we have
\[
  \E\left( \left| \sum_{k=1}^{n-1} \frac{\Phi_1^2}{a_{k+1}^2} \left( \frac{S_k}{k} \right)^2\right| \right)
  \leq \sum_{k=1}^{n-1} \E\left( \frac{1}{k^2|a_{k+1}|^2} \left| S_k \right|^2 \right) = o(n^{1 - 2 \Re(\Phi_1)})
\]
by \eqref{eqn:intermediate}. Therefore, we have
\[
  \lim_{n \to \infty} \frac{1}{n^{1 - 2 \Re(\Phi_1)}} \sum_{k=1}^{n-1} \frac{\Phi_1^2}{a_{k+1}^2} \left( \frac{S_k}{k} \right)^2 = 0 \text{ in probability,}
\]
by Markov inequality.

We then use the triangular and Jensen's inequalities to first bound
\begin{align*}
  \E\left( \left| \sum_{k=1}^{n-1} \frac{\Phi_2}{a_{k+1}^2}\frac{S_k^{(2)}}{k}\right| \right)
  &\leq \sum_{k=1}^{n-1} \frac{|\Phi_2|}{k|a_{k+1}|^2} \E\left( \left| S_k^{(2)} \right| \right)\\
  &\leq \sum_{k=1}^{n-1} \frac{1}{k|a_{k+1}|^2} \E\left(\left|S_k^{(2)}\right|^{2} \right)^{1/2}\\
  &\leq \sum_{k=1}^{n-1} \frac{|a^{(2)}_{k+1}|}{k|a_{k+1}|^2} \E\left(\left|\frac{S_k^{(2)}}{a^{(2)}_{k+1}}\right|^{2} \right)^{1/2} = O( n ^{\Re (\Phi_2) - 2 \Re (\Phi_1)}).
\end{align*}
By Lemma~\ref{lem:martingale} and the previous computations, we have

\[
  \lim_{n \to \infty} \frac{1}{n^{1 - 2 \Re(\Phi_1)}} \sum_{k=1}^{n-1} \frac{\Phi_2}{a_{k+1}^2}\left(\frac{S_k^{(2)}}{k}\right)^2 = 0 \text{ in probability,}
\]
by Markov inequality, and the lemma is proved.
\end{proof}

\subsection{Lindeberg's condition}
\label{subsec:lindeberg}
We now check the final condition for applying the central limit theorem is verified.
\begin{lemma}
\label{lem:lindeberg}
 We have, for all $\varepsilon>0$,
  \[
    \lim_{n \to \infty} \frac{1}{v_n} \sum_{k=1}^{n-1} \E\left( |\Delta M_k|^2 \mathbf{1}_{|\Delta M_k|\geq \varepsilon \sqrt{v_n}} \middle| \mathcal{F}_{k} \right)  = 0 \text{ in probability,}
  \]
  in other words, the Lindeberg's condition is satisfied.
\end{lemma}
\begin{proof}
First, we have the following inequality
  \begin{align*}
    \sum_{k=1}^{n-1} \E\left( |\Delta M_k|^2 \mathbf{1}_{|\Delta M_k|\geq \varepsilon \sqrt{v_n}} \middle| \mathcal{F}_{k} \right)
     & \leq \frac{1}{\varepsilon^2 v_n}\sum_{k=1}^{n-1} \E\left( |\Delta M_k|^4 \middle| \mathcal{F}_{k} \right).
  \end{align*}
Then, remark that
\begin{align*}
\E\left( |\Delta M_k|^4 \middle| \mathcal{F}_{k} \right) & = \frac{1}{|a_{k+1}|^4} \E\left( \left|\Delta S_k - \frac{\Phi_1}{k}S_k\right|^4 \middle| \mathcal{F}_{k} \right) \\
&  \leq \frac{1}{|a_{k+1}|^4} \E\left( \left(|\Delta S_k| + \left|\frac{\Phi_1}{k}S_k\right|\right)^4 \middle| \mathcal{F}_{k} \right)  \leq \frac{(1+ |\Phi_1|)^4}{|a_{k+1}|^4}.
\end{align*}
Using the previous computations and, again, Lemma~\ref{lem:martingale}, we find
\begin{align*}
    \sum_{k=1}^{n-1} \E\left( |\Delta M_k|^2 \mathbf{1}_{|\Delta M_k|\geq \varepsilon \sqrt{v_n}} \middle| \mathcal{F}_{k} \right)
     & \leq \frac{(1+ |\Phi_1|)^4}{\varepsilon^2 v_n}\sum_{k=1}^{n-1} \frac{1}{|a_{k+1}|^4} = O (n^{-2\Re(\Phi_1)}),
  \end{align*}
and finally
\[
  \lim_{n \to \infty} \frac{1}{v_n} \sum_{k=1}^{n-1} \E\left( |\Delta M_k|^2 \mathbf{1}_{|\Delta M_k|\geq \varepsilon \sqrt{v_n}} \middle| \mathcal{F}_{k} \right)  = 0 \text{ in probability,}
\]
  again by Markov inequality, and the lemma is proved.
\end{proof}

\begin{proof}{Proof of \eqref{eqn:mainSub}}
From Lemmas \ref{lem:quad-complex}, \ref{lem:quad-mod}, and \ref{lem:lindeberg}, and the identification $\IC \simeq \IR^2$, we prove \eqref{eqn:mainSub} by using Lindeberg’s Central Limit Theorem for vector martingales (see, e.g., \cite[Corollary 2.1.10]{Duflo1997}).
\end{proof}

The case $\Re(\Phi_1) = \frac{1}{2}$ is treated similarly, adapting the proof of the previous lemmas with the logarithmic rate. Indeed, since in this case $|a_n| \sim |C_1| n^{1/2}$, the equivalent of $(2.7)$ is $v_n \sim |C_1|^2 \log n$ such that $(|a_{n+1}| n)^{-2}\E\left( |S_n|^2\right) \sim |C_1|^2 \log n$. We refer to \cite{BercuLaulin2019} for a (slightly) more detailed example.

\section{Spiraling Elephant random walk}
\label{sec:spiraling}

We now turn to the study of the asymptotic properties of the ERW assuming that $\Re(\Phi_1)> \frac{1}{2}$. We use in this section a Poissonnization method adapted from the work of Janson \cite{Jan} for the study of urns. Specifically, we consider the ERW subordinated to an independent Yule process $(N_t, t \geq 1)$. We then use that $(S_{N_t}, t \geq 0)$ can be seen as an additive martingale of a branching Markov process.

In Section~\ref{subsec:cpxBMP}, we introduce the branching Markov process we will consider, and show its connection to the ERW in the plane. We then prove, in Section~\ref{subsec:martcv}, the almost sure convergence of $\erm^{-t \Phi_1} S_{N_t}$ as $t \to \infty$ by controlling the moments of this martingale. Finally, we study in Section~\ref{subsec:clt} the rate of convergence of this martingale by classical Fourier transform methods. We complete the section with a proof of \eqref{eqn:mainSup}.

\subsection{Associated complex-valued Markov branching process}
\label{subsec:cpxBMP}

Let us now describe the underlying branching process, constructed as a $\C$-valued Markov branching particle system. In this process, particles do not move or die out, but give birth at unit rate to new children, such that the position of each child of an individual at position $x$ is given by an independent copy of the random variable $x \erm^{i \theta}$. Let us precise the notation we use to construct this process.

The set of particles in this branching process is encoded using the classical Ulam-Harris-Neveu notation by
\[
  \mathcal{U} = \bigcup_{n \geq 0} \N^n \quad \text{ with the convention $\N^0 = \{\emptyset\}$}.
\]
The label $\emptyset$ corresponds to the initial particle of the branching process, alive at time $0$, its successive children are labelled $1,2,3$, etc. More generally, $u = (u(1),\ldots,u(k)) \in \N^k$ corresponds to the $u(k)$th child of the $u(k-1)$th child of ... of the $u(1)$th child of the initial particle. We write $|u|=k$ for the generation of $u$ (with convention $|\emptyset|=0$) and for all $j \leq k$, we write $u_j = (u(1),\ldots, u(j))$ the ancestor at the $j$th generation of $u$.

Let $(\Xi_u, u \in \mathcal{U})$ be i.i.d. Poisson point processes with unit intensity on $\R_+$, such that the atoms of $\Xi_u$ correspond to the ages at which particle $u$ gives birth to its children. More precisely, writing $(\xi_u(k), k \geq 1)$ for the sequence of atoms of $\Xi_u$ ranked in increasing order, we define for $u \in \mathcal{U}$
\begin{equation}
  T_u = \sum_{j=1}^{|u|} \xi_{u_{j-1}}(u(j)) \quad \text{ with the convention } T_\emptyset = \emptyset,
\end{equation}
corresponding to the birth-time of particle $u$ in the process. In other words, the birth time of the $j$th child of $u$ is given by $T_u + \xi_u(j)$ the sum of the birth-time of $u$ and its age at the creation of this child. The set of particles born before time $t$ is written $\mathcal{N}_t = \{u \in \mathcal{B} : T_u \leq t\}$ --which corresponds to the population alive at time $t$ since particles don't die.

Independently of the previous family, we introduce $(\theta_u, u \in \mathcal{U} \setminus\{\emptyset\})$ i.i.d. random variables with value in $[0,2\pi)$, that parametrize the displacement of $u$ with respect to its parent. The position of $u \in \mathcal{B}$, written $X_u$, is then defined by
\begin{equation}
  \label{eqn:position}
  X_u = \prod_{j=1}^{|u|} \erm^{i\theta_{u_j}} \quad \text{with the convention } X_\emptyset = 1.
\end{equation}
Let us mention that the branching Markov process $((X_u, u \in \mathcal{N}_t, t) \geq 0)$ takes values on the unit sphere $\{z \in \C : |z| =1\}$, and can be constructed as $((\erm^{i Y_u}, u \in \mathcal{N}_t), t \geq 0)$ where $Y$ is a continuous-time branching random walk on $\R$.

We introduce, for $k \in \N$, the quantity
\begin{equation}
  \label{eqn:defMartK}
  Z_k(t) = \sum_{u \in \mathcal{N}_t} X_u^k,
\end{equation}
which can be thought as the Fourier coefficients of the empirical measure $\sum_{u \in \mathcal{N}_t} \delta_{Y_t(u)}$. This additive functional of $X$ is connected to the ERW in the following lemma.
\begin{lemma}
\label{lem:obvious}
For all $n \geq 1$, we write
\[
  \tau_n = \inf\{t \geq 0 : \#\mathcal{N}_t = n\}.
\]
Let $(S_n, n \geq 1)$ be the ERW defined in the introduction, we have
\[
  (S_n, n \geq 1) \egaldistr \left( Z_{1}(\tau_n), n \geq 1 \right).
\]
\end{lemma}

\begin{proof}
This result is fairly straightforward and a classical observation in the study of reinforced processes. By construction, there are at time $\tau_n$ exactly $n$ particles in the branching Markov process $X$. Each particle gives birth, independently of one another, to a new child whose position will be given by a random rotation of angle $\theta$ of their own. Therefore, at time $\tau_{n+1}$ exactly one particle gives birth to a child, whose label  is uniformly chosen in $\mathcal{N}_{\tau_n}$.

As a result, we obtain that $Z_1(\tau_{n+1})$ is the sum of $Z_1(\tau_n)$ and a randomly rotated, uniformly chosen, term of the sum defining $Z_1(\tau_n)$. The evolution therefore matches the one of the ERW, which justifies the equality in distribution between the two processes.
\end{proof}

\begin{remark}
With a similar reasoning, and the notation of Lemma~\ref{lem:martingale}, observe that for all $k \geq 1$, we have
\[
 (S_n^{(k)}, n \geq 1) \egaldistr \left(Z_k(\tau_n), n \geq 1\right).\]
\end{remark}

To complete this section, we end by recalling that $N : t \mapsto \#\mathcal{N}_t$ is a standard Yule process, i.e. a continuous-time Markov process that jumps at rate $i$ from $i$ to $i+1$. We mention that $(\erm^{-t} N_t, t \geq 0)$ is a non-negative martingale, which verifies
\begin{equation}
  \label{eqn:sizePop}
  \lim_{t \to \infty} \erm^{-t} N_t = E \quad \text{a.s.}
\end{equation}
Recalling that $N_t$ has a Geometric distribution with parameter $\erm^{-t}$, we observe that $E$ is distributed as a unit-mean exponential random variable, hence is a.s. positive.

\subsection{Convergence of additive functional of the branching Markov process}
\label{subsec:martcv}

We prove in this section the following proposition, that allows us to describe the almost sure asymptotic behaviour of the branching random walk.
\begin{proposition}
\label{prop:cvMart}
If $\Re(\Phi_1) > 1/2$, the martingale $(\erm^{-\Phi_1t} Z_1(t), t \geq 0)$ converges a.s. and in $\L^2$ towards a complex random variable $W$, that satisfies
\begin{equation}
  \label{eqn:momentsOfW}
  \E(W) = 1, \quad \E(W^2) = \frac{2 \Phi_1}{2 \Phi_1 - \Phi_{2}} \quad \text{and} \quad \E(|W|^2) = \frac{2 \Re(\Phi_1)}{2 \Re(\Phi_1)-1}.
\end{equation}
\end{proposition}

\begin{remark}
\label{rem:mmts}
From \eqref{eqn:momentsOfW}, one can compute mean and covariance matrix of the vector $(A,B) := (\Re(W), \Im(W))$. We have $\E(A) = 1$, $\E(B) = 0$, and writing $\Sigma = \left( \begin{array}{cc}  \sigma^2 & \rho \\ \rho & \tau^2 \end{array} \right)$ the covariance matrix of that vector, we have
\begin{equation}
\begin{split}
  \sigma^2 = \frac{1}{2}\left( \frac{1}{2 \Re(\Phi_1)-1} + \Re\left(\frac{\Phi_2}{2 \Phi_1 - \Phi_{2}}  \right)\right), \quad \tau^2 &= \frac{1}{2}\left(\frac{1}{2  \Re(\Phi_1)-1} - \Re\left(\frac{\Phi_2}{2 \Phi_1 - \Phi_{2}}  \right)\right)\\
  \text{and }\quad \rho &= \Im\left(\frac{\Phi_1}{2 \Phi_1 - \Phi_{2}}\right).
\end{split}
\end{equation}
\end{remark}

To prove Proposition~\ref{prop:cvMart}, it will be sufficient to show that $(\erm^{-\Phi_1 t} Z_1(t), t \geq 0)$ is a martingale and is bounded in $\L^2$, then to study the asymptotic behaviour of the moments of this process. To this end, we compute the first two moments of this additive functional.
\begin{lemma}
\label{lem:firstMoment}
For all $k \in \N$, we have $\E(Z_k(t)) = \erm^{t \Phi_k}$.
\end{lemma}

\begin{proof}
Let us first observe that $Z_k$ is integrable, as by triangular inequality, we have $|Z_k(t)| \leq N_t$, with $\E(N_t) = \erm^t < \infty$.

We write $u(t) = \E(Z_k(t))$. We compute this quantity by decomposing the branching process at the first splitting time. Recall that $\tau_2$ the first branching time is distributed as a standard Exponential random variable, and $X_1$ is the position of the first child of the initial individual, we observe that on $\{\tau_2< t\}$, we have
\begin{equation}
  \label{eqn:branchSplitting}
  Z_k(t) \egaldistr Z'_k(t -\tau_2) + X_1^k Z''_k(t-\tau_2),
\end{equation}
where $Z'_k$ and $Z''_k$ are two independent copies of $Z_k$. This formula follows from the branching property (evolution of particles after their birth is independent from the rest of the process) and the product formula for the position of particles \eqref{eqn:position}.

This equation therefore implies
\begin{align*}
  u(t) &= \IP(\tau_2 > t) + \E(1 + X_1^k)\E\left(Z'_k(t-\tau_1) \ind{t > \tau_1} \right)\\
  &= \erm^{-t} + (1 + \Phi_k) \int_0^t \erm^{-s} u(t-s) \dd s = \erm^{-t} + (1+\Phi_k) \erm^{-t} \int_0^t \erm^r u(r) \dd r.
\end{align*}
As a consequence, $t \mapsto \erm^t u(t)$ is a solution of the linear differential equation $y' = (1 + \Phi_k) y$, with initial condition $u(0)=1$, from which we immediately complete the proof.
\end{proof}

The second moment of $Z_k$ is computed in a similar fashion. This result should be compared with Lemmas~\ref{lem:quad-complex} and \ref{lem:quad-mod} in the previous section.
\begin{lemma}
\label{lem:secondMoment}
Let $t \geq 0$ and $k \geq 1$, we have
\[
  \E(|Z_k(t)|^2) = \begin{cases}
   \frac{2 \Re(\Phi_k)}{2 \Re(\Phi_k)-1} \erm^{2 \Re(\Phi_k) t} - \frac{1}{2\Re(\Phi_k) - 1} \erm^t  & \text{if $\,\Re(\Phi_k) \neq 1/2$}\\
   (t+1) \erm^{t}  & \text{if $\,\Re(\Phi_k) = 1/2$},
  \end{cases}
\]
as well as
\[
  \E(Z_k(t)^2) = \begin{cases}
   \frac{2 \Phi_k}{2 \Phi_k - \Phi_{2k}} \erm^{2 \Phi_k t} - \frac{\Phi_{2k}}{2 \Phi_k - \Phi_{2k}} \erm^{\Phi_{2k}t} & \text{if $\,2 \Phi_k \neq \Phi_{2k}$}\\
   (1 + 2 \Phi_k t)\erm^{2 \Phi_k t} & \text{if $\,2 \Phi_k = \Phi_{2k}$.}
  \end{cases}
\]
\end{lemma}

\begin{proof}
We compute these second moments in the same way we computed the first, by using \eqref{eqn:branchSplitting} to identify a differential equation they satisfy. Indeed, decomposing at the first branching time, we have
\begin{align*}
  v(t) &:= \E(|Z_k(t)|^2) = \E(Z_k(t) \bar{Z_k(t)})\\
  &= \erm^{-t} + \int_0^t \erm^{-s} \E\left( \big( Z'_k(t-s) + X_1^k Z''_k(t-s) \big)  \bar{\left(  Z'_k(t-s) + X_1^k Z''_k(t-s)\right)}\right)  \dd s\\
  &= \erm^{-t} + \int_0^t \erm^{-s} 2 v(t-s) + 2 \Re(\Phi_k) |\E(Z_k(t-s))|^2 \dd s.
\end{align*}
As a result, using Lemma~\ref{lem:firstMoment} with time reversal in the integral, we have
\[
  \erm^t v(t) = 1 + \int_0^s 2 \erm^s \left( v(s) + 2 \Re(\Phi_k) \left| \erm^{\Phi_k s} \right|^2 \right) \dd s.
\]
Therefore, $t \mapsto \erm^{t}v(t)$ is the solution of $y'(t) = 2 y(t) + 2 \Re(\Phi_k) \erm^{1 + 2 \Re(\Phi_k) t}$ with initial condition $y(0) = 1$. As a result, we conclude that, if $\Re(\Phi_k) \neq \frac{1}{2}$, we have
\[
  v(t) = \frac{2 \Re(\Phi_k)}{2 \Re(\Phi_k)-1} \erm^{2 \Re(\Phi_k) t} - \frac{1}{2\Re(\Phi_k) - 1} \erm^t
\]
while $v(t) = (t+1) \erm^t$ if $\Re(\Phi_k) = \frac{1}{2}$, completing the first part of the proof.

The second part of the proof follows in an identical fashion. We have
\begin{align*}
  w(t) &:= \E(Z_k(t)^2) \\
  &= \erm^{-t} + \int_0^t \erm^{-s} \E\left( \left( Z'_k(t-s) + X_1^k Z''_k(t-s) \right)^2\right)  \dd s\\
  &= \erm^{-t} + \int_0^t \erm^{-s} \left((1 + \Phi_{2k}) w(t-s) + 2 \Phi_k \E(Z_k(t-s))^2\right) \dd s.
\end{align*}
Using again Lemma~\ref{lem:firstMoment} with time reversal in the integral, we have
\[
  \erm^t w(t) = 1 + \int_0^t (1 + \Phi_{2k}) \erm^s v(s) + 2 \Phi_k \erm^{s (1 + 2 \Phi_k)} \dd s,
\]
which solving the differential equation yields
\[
  w(t) = \frac{2 \Phi_k}{2 \Phi_k - \Phi_{2k}} \erm^{2 \Phi_k t} - \frac{\Phi_{2k}}{2 \Phi_k - \Phi_{2k}} \erm^{\Phi_{2k}t}
\]
if $\Phi_{2k} \neq 2 \Phi_k$, while $w(t) = (1 + 2 \Phi_k t)\erm^{2 \Phi_k t}$ otherwise.
\end{proof}

We are now able to prove the main result of the section.
\begin{proof}[Proof of Proposition~\ref{prop:cvMart}]
Using the branching property, we observe that for all $s \leq t$, we have
\[
  Z_1(t+s) \egaldistr \sum_{u \in \mathcal{N}_t} X_u Z^{(u)}_1(s),
\]
where $(Z^{(u)}_1(s), u \in \mathcal{N}_t)$ are i.i.d. copies of $Z_1(s)$ that are independent from the branching Markov process. Therefore, writing $(\mathcal{F}_t)$ the natural time filtration of $X$, we have
\[
  \E(Z_1(t+s) | \mathcal{F}_t) = Z_1(t) \E(Z_1(s)) = Z_1(t) \erm^{\Phi_1 s},
\]
from which we deduce that $(\erm^{-\Phi_1 t} Z_1(t), t \geq 0)$ is a martingale.

Let us now assume that $\Re(\Phi_1)> \frac{1}{2}$. In this situation, by Lemma~\ref{lem:secondMoment}, we have
\[
  \lim_{t \to \infty} \E\left( \left|Z_1(t) \erm^{-\Phi_1 t}\right|^2 \right) = \frac{2 \Re(\Phi_1)}{2 \Re(\Phi_1) - 1},
\]
which shows in particular that the martingale is bounded in $\L^2$, therefore that it converges a.s. and in $\L^2$ to a random variable $W$. Moreover, as $2 \Re(\Phi_1) > 1 > \Re(\Phi_2)$, we also obtain
\[
  \lim_{t \to \infty} \E\left( \left(Z_1(t) \erm^{-\Phi_1 t}\right)^2 \right) = \frac{2 \Phi_1}{2 \Phi_1 - 1}.
\]
We use these two convergences to identify the moments of $W$.
\end{proof}

\subsection{Rate of convergence of the additive martingale}
\label{subsec:clt}

We assume in this section that $\Re(\Phi_1) > 1/2$, therefore that $\erm^{-\Phi_1 t} Z_1(t)$ converges a.s. to a non-degenerate random variable $W$ as proved in Proposition~\ref{prop:cvMart}. The aim of this section is to determine the rate of this convergence. More precisely, we write
\[
  R_t = \erm^{-t/2} \left( Z_1(t) - \erm^{\Phi_1 t} W \right),
\]
we show in the next proposition that conditionally on $E$ (the almost-sure limit of $\erm^{-t}\#\mathcal{N}_t$ defined in \eqref{eqn:sizePop}), $R_t$ converges in law to a complex normal random variable.

\begin{proposition}
\label{prop:mixedCLT}
Under the previous assumptions and notation, for all $(\lambda,\mu) \in \R^2$, we have
\[
  \lim_{t \to \infty} \E\left( \erm^{i \lambda \Re(R_{t}) + i \mu \Im(R_{t})} \middle| \mathcal{F}_t \right) = \exp\left( - \frac{E}{4}  \frac{1}{2 \Re(\Phi_1) - 1} (\lambda^2+ \mu^2) \right) \quad \text{a.s.}
\]
\end{proposition}

Before turning to the proof of this result, let us mention as an immediate consequence the joint convergence in distribution of $R_t$ and $\erm^{-t}\mathcal{N}_t$.

\begin{corollary}
\label{cor:obvious}
Assuming that $\Re(\Phi_1) > 1/2$, let $G\sim \mathcal{N}_\C(0, \frac{1}{2 \Re(\Phi_1) - 1})$ independent of the branching Markov process, we have
\[
  \lim_{t \to \infty} (\erm^{-t} \#\mathcal{N}_t, R_t) = (E, \sqrt{E}N) \quad \text{ in distribution}.
\]
\end{corollary}

This result is immediate by computing the joint Fourier transform of this random vector, and applying dominated convergence theorem.

\begin{proof}[Proof of Proposition~\ref{prop:mixedCLT}]
We first remark, using again the branching property that for all $s,t \geq 0$,
\[
  \erm^{-\Phi_1 (t+s)} Z_1(t+s) \egaldistr \erm^{- \Phi_1 t}\sum_{u \in \mathcal{N}_t} X_u Z_1^{(u)}(s)\erm^{-\Phi_1 s}.
\]
Therefore, letting $s \to \infty$ in the above expression, we conclude that for all $t \geq 0$,
\begin{equation}
  \label{eqn:smoothingTransform}
  \erm^{\Phi_1 t} W \egaldistr \sum_{u \in \mathcal{N}_t} \erm^{\Phi_1 t} X_u W^{(u)},
\end{equation}
where $(W^{(u)}, u \in \mathcal{N}_t)$ are i.i.d. copies of $W$ further independent of $\mathcal{F}_t$. In particular, we can rewrite
\[
  R_{t} = \sum_{u \in \mathcal{N}_t} \erm^{-t/2}X_u (1 - W^{(u)}).
\]

As a consequence,
\begin{align*}
  \lambda \Re(R_t) + \mu \Im(R_t)
  &\egaldistr \sum_{u \in \mathcal{N}_t} \lambda (\Re(X_u) \Re(1-W^{(u)}) - \Im(X_u) \Im(1-W^{(u)}))\\
  &\qquad \qquad \qquad + \mu (\Re(X_u) \Im(1-W^{(u)}) + \Im(X_u) \Re(1-W^{(u)})) \\
  &= \sum_{u \in \mathcal{N}_t} \Re(1-W^{(u)}) (\lambda \Re(X_u) + \mu \Im(X)) + \Im(1-W^{(u)})(-\lambda \Im(X_u) + \mu \Re(X_u)),
\end{align*}
yielding
\begin{equation}
  \label{eqn:fourierTransform}
  \E\left( \erm^{i \lambda \Re(R_{t}) + i \mu \Im(R_{t})} \middle| \mathcal{F}_t \right) = \prod_{u \in \mathcal{N}_t} \phi(\erm^{-t/2} (i\lambda + \mu) X_u),
\end{equation}
where, for $z \in \C$, we write
\[
  \phi(z) = \E\left( \erm^{i \Re(z) \Re(1-W) + i \Im(z) \Im(1-W)} \right).
\]

Using the first two moments of W (see Proposition~\ref{prop:cvMart} and Remark \ref{rem:mmts}), we observe that for all $M > 0$, uniformly in $|z| \leq M$ we have
\[
  \phi(z \erm^{-t/2}) = 1 - \frac{\erm^{-t}(1+o(1))}{2}\left(\Re(z)^2 \sigma^2 + 2 \Re(z) \Im(z) \rho + \Im(z)^2 \tau^2\right).
\]
Let $(\lambda,\mu) \in \R^2$, we write $z = \frac{i \lambda + \mu}{\sqrt{\lambda^2 + \mu^2}}$, so that $|z| = 1$. Then, for all $u \in \mathcal{N}_t$, we have $|z X_u| = 1$, which allows us to use the double-angle formulae:
\[
  \Re(z X_u)^2 = \frac{1 + \Re((zX_u)^2)}{2}, \quad \Im(z X_u)^2 = \frac{1 - \Re((zX_u)^2)}{2} \quad \text{and} \quad \Re(z X_u)\Im(zX_u) = \frac{\Im((zX_u)^2)}{2}.
\]
Consequently, we obtain from \eqref{eqn:fourierTransform}
\begin{multline}
  \E\left( \erm^{i \lambda \Re(R_{t}) + i \mu \Im(R_{t})} \middle| \mathcal{F}_t \right)\\
  = \exp\left( \frac{\erm^{-t}(1+o(1))}{4}(\lambda^2 + \mu^2) \left((\sigma^2 + \tau^2)\#\mathcal{N}_t + (\sigma^2 - \tau^2)\Re(z Z_2(t)) + 2 \rho \Im(z Z_2(t)) \right) \right).
\end{multline}

Let us now recall from \eqref{eqn:Nunidimensional} that $\Re(\Phi_2) < 1$, thus from a straightforward application of Lemma~\ref{lem:secondMoment}, we deduce that $Z_2(t)\erm^{-t}$ converges to $0$ in probability. As a result, we obtain
\[
  \lim_{t \to \infty} \E\left( \erm^{i \lambda \Re(R_{t}) + i \mu \Im(R_{t})} \middle| \mathcal{F}_t \right) = \exp\left( - \frac{E}{4}(\lambda^2 + \mu^2) (\sigma^2 + \tau^2) \right).
\]
The proof is now complete using that $\sigma^2 + \tau^2 = \Var(\Re(W_u)) + \Var(\Im(W_u)) = \E(|W_u|^2) - 1$ and Proposition~\ref{prop:cvMart}.
\end{proof}

We now complete the proof of Theorem~\ref{thm:main}.

\begin{proof}[Proof of \eqref{eqn:mainSup}]
Using Lemma~\ref{lem:obvious}, we can write
\[
  \frac{S_n}{n^{\Re(\Phi_1)}}\erm^{-i \Im(\Phi_1) \log n} = Z_1(\tau_n) \exp\left( - \Phi_1 \log N_{\tau_n} \right),
\]
recalling that $N_t = \#\mathcal{N}_t$, hence $N_{\tau_n} = n$ a.s. By Proposition~\ref{prop:cvMart} and \eqref{eqn:sizePop}, we have
\[
  Z_1(t) \erm^{-\Phi_1 \log N_t} = Z_1(t) \erm^{-\Phi_1 t} \exp(-\Phi_1 \log (N_t \erm^{-t})) \to W \exp(-\Phi_1 \log E) \quad \text{a.s.}
\]
Using that $\tau_n\to \infty$ a.s., we conclude that
\[
  \lim_{n \to \infty}   \frac{S_n}{n^{\Re(\Phi_1)}}\erm^{-i \Im(\Phi_1) \log n} = W \exp(-\Phi_1 \log E) \quad \text{a.s.}
\]

We now turn to the central limit theorem result, writing
\begin{align*}
  &\phantom{=}\frac{S_n - n^{\Re(\Phi) \erm^{i \Im(\Phi) \log n}}  W \exp(-\Phi_1 \log E) }{\sqrt{n}}\\
  &= \frac{Z_1(\tau_n) - W \exp(\Phi_1 \log (N_{\tau_n}/E))}{\sqrt{n}}\\
  &= \erm^{-\tau_n/2}\frac{Z_1(\tau_n) - W \erm^{\Phi_1 \tau_n} + W \erm^{\Phi_1 \tau_n} - W \exp(\Phi_1 \log (N_{\tau_n}/E)}{\sqrt{N_{\tau_n} \erm^{-\tau_n}}}.
\end{align*}
By Corollary~\ref{cor:obvious}, we observe that
\[
  \lim_{t \to \infty} \erm^{-t/2}\frac{Z_1(t) - W \erm^{\Phi_1 t}}{\sqrt{N_t \erm^{-t}}} = \mathcal{N}_\C(0,\frac{1}{2 \Re(\Phi_1) - 1}) \quad \text{in law},
\]
from which we deduce straightforwardly that $\erm^{-\tau_n/2}\frac{Z_1(\tau_n) - W \erm^{\Phi_1 \tau_n} }{\sqrt{N_{\tau_n} \erm^{-\tau_n}}}$ admit the same limit in distribution as $n \to \infty$, using that $\tau_n - \log n$ is tight. In addition, we have
\[
  \frac{\erm^{\Phi_1 t} - \erm^{\Phi_1 \log (N_t / E)}}{\sqrt{N_t}} = \frac{1+o(1)}{\sqrt{E}}\erm^{(\Phi_1 - 1/2)t}(1 - \erm^{\Phi_1 \log (N_t \erm^{-t}/E)}) \quad \text{a.s.}
\]
As $ 1/2 < \Re(\Phi_1)<1$ and $\lim_{t \to \infty} \frac{N_t - E \erm^{t}}{\erm^{t/2}}$ converges in distribution towards a Gaussian random variable, we conclude that this quantity converges to $0$ in probability.

As a result, from Slutsky's lemma, we deduce the second part of \eqref{eqn:mainSup}.
\end{proof}

\end{document}